\documentclass[12pt]{article}

\usepackage{amsmath}
\usepackage{amssymb}
\usepackage{geometry}
\usepackage{amsthm}
\usepackage{abstract}
\usepackage{hyperref}

\newtheorem{lemma}{Lemma}
\newtheorem{theorem}{Theorem}
\newtheorem{proposition}{Proposition}

\newtheorem{assumptions}{Assumptions}
\newtheorem{corollary}{Corollary}

\title{\vspace{-2em}Exponential decay estimates for the resolvent kernel on a Riemannian manifold}
\author{Zhirayr Avetisyan\footnote{Department of Mathematics: Analysis, Logic and Discrete Mathematics, Ghent University, Belgium. zhirayr.avetisyan@ugent.be, contact@z-avetisyan.com}}

\date{}

\begin{document}
\maketitle
\vspace{-2em}
\begin{abstract} On a complete, connected, non-compact Riemannian manifold, with Ricci curvature bounded from below, we establish exponential decay estimates at infinity for the spherical sums of the resolvent kernel, i.e., the integral kernel of the resolvent $(-\Delta+\lambda)^{-1}$ of the Laplace-Beltrami operator for $\lambda>0$. The exponential decay rate in these estimates is optimal, in the sense that it cannot be improved based solely on a uniform bound on the Ricci curvature.

In addition to technical results, the paper offers two conceptual takeaways. Firstly, in pursuit of optimal decay estimates, it may be worth shifting the emphasis from the resolvent kernel to its density - a geometrically more natural object. Secondly, contrary to expectations, the exponential decay of the resolvent kernel is driven mainly by the area/volume growth, and the bottom of the spectrum of $\Delta$ is not a decisive factor. This is in contrast with the heat kernel decay, and shows that optimal resolvent kernel estimates cannot be derived from heat kernel estimates.
\end{abstract}

\tableofcontents

\section{Introduction}

It is known among specialists, that on Riemannian manifolds with exponential volume growth, the rate of the exponential growth interferes with the long-distance decay of the heat and resolvent kernels. While this has been observed and discussed by many authors, the precise way of that interference does not appear to be known beyond very narrow, explicitly tractable classes of manifolds. The present paper is devoted to the study of how the volume growth and the spectral parameter come together to determine the exponential decay of the spherical sums of the resolvent kernel. Our main assumption is a lower bound on the Ricci curvature, and we argue that, without additional information, our exponential upper bounds cannot be improved; they coincide with exponential lower bounds for a certain class of manifolds. In particular, for manifolds with perfectly isotropic mean spherical curvature (i.e., harmonic spaces or open models), our spherical mean estimates automatically produce pointwise estimates.

For a complete, connected, non-compact Riemannian manifold $(M,g)$ of dimension $n>1$, let $\Delta=\operatorname{div}\nabla$ be the Laplace-Beltrami operator. All Lebesgue spaces $L^q(M)$ are understood with respect to the Riemannian volume $\mathrm{vol}_g$. It is well-known that $\Delta$ is essentially self-adjoint in $L^2(M)$ with the core $C^\infty_c(M)$, and the spectrum of $\Delta$ is non-positive, $\sigma(\Delta)\subset(-\infty,0]$. Denote
$$
E\doteq-\sup\sigma(\Delta)\in[0,+\infty).
$$
For $\lambda>0$, the resolvent $(-\Delta+\lambda)^{-1}$ is a bounded operator on $L^2(M)$, and is an integral operator with a strictly positive integral kernel
\begin{equation}
K_\lambda\in C^\infty(M\times M\setminus\operatorname{diag}(M\times M),\mathbb{R}_+),\label{KlambdaSmooth}
\end{equation}
where $\mathbb{R}_\pm\doteq\pm(0,+\infty)$. In fact, it is known that
\begin{equation}
K_\lambda(x,y)=\int\limits_0^{+\infty}e^{-\lambda t}p(x,y,t)dt,\quad\forall x,y\in M,\quad\forall\lambda\in\mathbb{R}_+\label{KpRel}
\end{equation}
where $p\in C^\infty(M\times M\times\mathbb{R}_+)$ is the heat kernel on $(M,g)$ (for the basic facts about the heat kernel on Riemannian manifolds, see the Section 7.3 in Grigor'yan's monograph \cite{Gri09}). From the stochasticity of the heat kernel (which holds in amazing generality, see Theorem 11.8 in \cite{Gri09}), namely,
$$
\int\limits_Mp(z,y,t)d\mathrm{vol}_g(z)=\int\limits_Mp(x,z,t)d\mathrm{vol}_g(z)=1,\quad\forall x,y\in M,
$$
it follows that
\begin{equation}
\int\limits_MK_\lambda(z,y)d\mathrm{vol}_g(z)=\int\limits_MK_\lambda(x,z)d\mathrm{vol}_g(z)=\frac1\lambda,\quad\forall x,y\in M,\quad\forall\lambda\in\mathbb{R}_+.\label{IntKEq}
\end{equation}
Indeed,
$$
\int\limits_M\int\limits_0^{+\infty}e^{-\lambda t}p(z,y,t)dtd\mathrm{vol}_g(z)=\int\limits_0^{+\infty}e^{-\lambda t}\int\limits_Mp(z,y,t)d\mathrm{vol}_g(z)dt=\frac1\lambda.
$$
Let $d:M\times M\to[0,+\infty)$ be the geodesic distance function (metric) on $(M,g)$. It is clear from (\ref{IntKEq}) that the resolvent kernel $K_\lambda(x,y)$ must have a certain decay as $d(x,y)\to+\infty$, at least in order to guarantee the convergence of the integral. However, to the best of our knowledge, the most accurate decay estimates of the resolvent kernel on general Riemannian manifolds known in the literature do \textbf{not} suffice to ensure integrability.

One approach to the resolvent kernel $K_\lambda$ is to address it directly. For $\lambda>0$, the operator $-\Delta+\lambda$ is coercive, and $K_\lambda$ is a Green's function for it. Thus, if $(M,g)$ has bounded geometry, then by Remark 2.1 in \cite{Anc87}, we have
$$
K_\lambda(x,y)\le c\,e^{-\alpha d(x,y)}
$$
for $d(x,y)>1$. Here the exponential coefficient $\alpha$ is unknown, and if the manifold has exponential volume growth, it is unclear if this estimate guarantees integrability. But one can obtain an explicit formula for $\alpha$ following the general scheme by Cheeger, Gromov and Taylor \cite{CGT82}. Namely, applying Corollary 3.2 in \cite{CGT82} to the function $f(z)=(z^2+E+\lambda)^{-1}$ (in view of the remark at the bottom of page 39 of the same paper), one arrives at the estimate
$$
K_\lambda(x,y)\le c_\alpha\,e^{-\alpha d(x,y)},\quad\forall\alpha\in\left(0,\sqrt{E+\lambda}\right),
$$
for $d(x,y)>r_0>0$. This decay estimate does not guarantee the integrability of the kernel $K_\lambda$ if the volume growth is exponential, as $E\le\frac{\mu^2}4$ (see more details below). In fact, we are not aware of any integrable exponential decay estimates for the Green's function in the literature with an explicit coefficient, applicable to the Laplace-Beltrami operator on sufficiently generic complete Riemannian manifolds, e.g., those of bounded geometry.

Another way to derive decay properties for the resolvent kernel is to utilise the known decay properties of the heat kernel, and use the relation (\ref{KpRel}). The archetypical example is the Euclidean heat kernel in $\mathbb{R}^n$,
$$
p(x,y,t)=\frac1{(4\pi t)^{\frac{n}2}}e^{-\frac{d(x,y)^2}{4t}}.
$$
Plugging into (\ref{KpRel}), and using formulae (\ref{Intabc}) and (\ref{BesselAsympBig}) in the Appendix, we find that
\begin{equation}
K_\lambda(x,y)=2\left(\frac{2\sqrt{\lambda}}r\right)^{\frac{n}2-1}\mathrm{K}_{\frac{n}2-1}(\sqrt{\lambda}r)\le c_n\frac{\lambda^{\frac{n-3}4}}{r^{\frac{n-1}2}}e^{-\sqrt{\lambda}r}\label{RnK}
\end{equation}
for $x,y\in M$ such that $r=d(x,y)>\frac1{\sqrt{\lambda}}$ (here $\mathrm{K}_{\frac{n}2-1}$ is the modified Bessel function of the second kind). Note that, since the volume growth in the Riemannian manifold $M=\mathbb{R}^n$ is polynomial, this exponential decay is sufficient to guarantee the convergence of the integral in (\ref{IntKEq}).

The situation for Riemannian manifolds with superpolynomial volume growth is more subtle. In this context, the best known example is the $n=m+1$-dimensional hyperbolic space $\mathbb{H}^{m+1}$, where we have the two-sided uniform estimate
\begin{equation}
p(x,y,t)\simeq_m\frac{(1+d(x,y)+t)^{\frac{m}2-1}(1+d(x,y))}{t^{\frac{m+1}2}}e^{-\frac{m^2}4t-\frac{d(x,y)^2}{4t}-\frac{m}2d(x,y)},\label{DaviesMandouvalos}
\end{equation}
see Theorem 3.1 in \cite{DaMa88}. Since $(1+r+t)^{\frac{m}2-1}\simeq_m(1+r)^{\frac{m}2-1}+t^{\frac{m}2-1}$, with the help of the same relations (\ref{KpRel}) and (\ref{Intabc}), we arrive at the uniform two-sided estimate
$$
K_\lambda(x,y)\simeq_m r^{\frac12}\left(1+\frac1r\right)^{\frac{m}2}\left(\frac{m^2}4+\lambda\right)^{\frac{m-1}4}e^{-\frac{m}2r}\,\mathrm{K}_{\frac{m-1}2}\left(\sqrt{\frac{m^2}4+\lambda}\,r\right)
$$
\begin{equation}
+\left(1+\frac1r\right)e^{-\left[\frac{m}2+\sqrt{\frac{m^2}4+\lambda}\,\right]\,r},\label{HnK}
\end{equation}
for $x,y\in M$ and $r=d(x,y)$. With the expansion (\ref{BesselAsympBig}) we obtain the estimate
\begin{equation}
K_\lambda(x,y)\le\left(c_m\left(\frac{m^2}4+\lambda\right)^{\frac{m-2}4}+d_m\right)e^{-\left[\frac{m}2+\sqrt{\frac{m^2}4+\lambda}\right]\,r},\label{HnK1}
\end{equation}
for $x,y\in M$ such that $r=d(x,y)>(\frac{m^2}4+\lambda)^{-\frac12}$. In this case, the volume growth is $V(r)\sim\exp(mr)$, so that the above decay is sufficient to secure the convergence of the integral in (\ref{IntKEq}) for $\lambda>0$.

A slight generalisation of the hyperbolic spaces $\mathbb{H}^{m+1}$ are what Mazzeo and Melrose \cite{MaMe87} call asymptotically hyperbolic manifolds, i.e., simply connected Riemannian manifolds with sectional curvatures becoming constant near the conformal boundary. The authors of \cite{MaMe87} provide many analytical properties of the heat kernel, but not precise decay estimates. The latter was done in the recent paper \cite{ChHa20} by Chen and Hassel under additional assumptions of non-positivity of sectional curvatures (Cartan-Hadamard condition) and spectral conditions on the Laplace-Beltrami operator $\Delta$. Under these conditions, the authors provide an exponential decay estimate for the heat kernel similar to (\ref{DaviesMandouvalos}) in Proposition 19. They also give an exponential decay estimate $e^{-\frac{m}2r}$ for the resolvent kernel (Proposition 16), which misses the second, larger contribution $e^{-\sqrt{\frac{m^2}4+\lambda}\,r}$, and thus does not guarantee integrability. However, their optimal estimate for the heat kernel in Proposition 19 easily gives the full decay of the resolvent kernel, as in the classical hyperbolic case (\ref{HnK}).

A further line of generalisation from hyperbolic spaces is represented by non-compact Riemannian symmetric spaces of rank one. In this case, the rich and very rigid structure of the semisimple Lie theory allows for a detailed analysis and explicit constructions, yielding heat kernel estimates essentially of the same form as in the hyperbolic case \cite{AnJi99}. There are also similar estimates of the heat kernel on the so-called Damek-Ricci harmonic spaces \cite{ADY96}. All these estimates, in nature identical to the Davies-Mandouvalos estimates for the harmonic space, produce resolvent kernel estimates of the form (\ref{HnK1}), which is rather satisfactory.

For general Riemannian manifolds, however, we do not have such precise heat kernel upper bounds. One of the sharpest such results to date that we know of, for a sufficiently large class of manifolds (those satisfying a certain isoperimetric inequality), is Theorem 5.3 in Grigor'yan's paper \cite{Gri94}, which gives
\begin{equation}
p(x,y,t)\le\frac{c_\gamma}{t^{\frac{n}2}}e^{-\gamma Et-\gamma\frac{d(x,y)^2}{4t}-\bar c_\gamma\sqrt{E}d(x,y)},\label{GrigoryanEst}
\end{equation}
for every $\gamma\in(0,1)$. This results in a resolvent kernel estimate
$$
K_\lambda(x,y)\le c_\gamma\,2^{\frac{n}2}\left(\frac{\sqrt{\gamma E+\lambda}}{\sqrt{\gamma}r}\right)^{\frac{n}2-1}e^{-\bar c_\gamma\sqrt{E}r}\,\mathrm{K}_{\frac{n}2-1}\left(\sqrt{\gamma(\gamma E+\lambda)}\,r\right)
$$
\begin{equation}
\le c_{n,\gamma}\frac{(\lambda+\gamma E)^{\frac{n-3}4}}{r^{\frac{n-1}2}}e^{-\left[\bar c_\gamma\sqrt{E}+\sqrt{\gamma(\gamma E+\lambda)}\,\right]\, r}.\label{PostGrigoryanEst}
\end{equation}
This estimate falls short of ensuring the convergence of the integral in (\ref{IntKEq}) for two reasons. Firstly, the constant $\bar c_\gamma$ above is completely unknown, and comparing with the estimate for the hyperbolic space $\mathbb{H}^{n+1}$, we see that $\bar c_\gamma\le1$ should be expected. Secondly, in a wide class of manifolds we have $E\le\frac{\mu^2}4$, where $V(r)\sim\exp(\mu r)$, see Corollary 5.4.3 in \cite{Dav89a}, as well as Theorem 3.11 and Theorem 3.12 in the survey article \cite{Ura93}. If $\lambda>0$ is sufficiently small, then it is only in case of $\bar c_\gamma\xrightarrow[\gamma\to1-]{}1$ and $E=\frac{\mu^2}4$ that we can chose $\gamma\in(0,1)$ so that the integral in (\ref{IntKEq}) converges. And these conditions appear to be beyond our control.

Thus, even the best available pointwise upper bounds on the heat kernel for general Riemannian manifolds do not produce resolvent kernel estimates good enough to ensure the convergence of the integral (\ref{IntKEq}). To make matters worse, it turns out that even optimal heat kernel estimates need not result in optimal resolvent kernel estimates; see the discussion in Section \ref{DiscussionSection}. This clearly indicates that there is more to the decay of the resolvent kernel than what is known today.

The above discussion of the special cases, and of Grigor'yan's upper bounds, point to a hypothetical estimate of the form
\begin{equation}
K_\lambda(x,y)\le c_\lambda(r)e^{-\left[\frac\mu2+\sqrt{\frac{\mu^2}4+\lambda}\right]\,r}\label{WishfulEstimate}
\end{equation}
for $x,y\in M$ with $r=d(x,y)\gg 1$, where $c_\lambda$ is a positive rational function of $r$. Whether this kind of pointwise upper bounds hold for general manifolds seems to be very hard to check; in fact, there are grounds to believe that in full generality, pointwise upper estimates should not be expected to be very informative. See the discussion in Section \ref{DiscussionSection} below.

However, if we assume for a moment that the estimate (\ref{WishfulEstimate}) is true, then for the spherical sums (i.e., integrals over metric spheres) of the resolvent kernel we find the decay estimates
\begin{equation}
\int\limits_{d(x,y)=r}K_\lambda(x,y)dA_{r,y}(x)\le c_\lambda(r)e^{-\left[-\frac\mu2+\sqrt{\frac{\mu^2}4+\lambda}\right]\,r},\label{SphericalSumsEstimate}
\end{equation}
because the area of the sphere may grow as $\exp(\mu r)$.

The principal result of this paper is the observation, that upper bounds of the spherical sums of the kind (\ref{SphericalSumsEstimate}) \textbf{do} hold, independently of the validity of pointwise estimates (\ref{WishfulEstimate}). This is the content of Theorem \ref{MainTheorem} below. All we need to know about the resolvent kernel $K_\lambda$ is (\ref{KlambdaSmooth}), (\ref{IntKEq}) and
\begin{equation}
\Delta_xK_\lambda(x,y)=\Delta_yK_\lambda(x,y)=\lambda K_\lambda(x,y),\quad\forall(x,y)\in M\times M\setminus\operatorname{diag}(M\times M).\label{KEigenfunc}
\end{equation}

In this work we are interested only in the rates of exponential decay at infinity. Any sub-exponential contributions are of secondary interest and not optimal generally.

\section{Preliminary analysis}

\subsection{The geometric setting}

We work with an $n$-dimensional connected, smooth, non-compact, complete Riemannian manifold $(M,g)$ without boundary, $n\in\mathbb{N}+1$. In what follows, $\mathrm{vol}_g$ will stand for the metric volume, and $d(x,y)$ will be the metric distance between points $x,y\in M$. Metric balls and spheres in $M$ will be denoted by
$$
\mathcal{B}_r(x)\doteq\left\{y\in M\;\vline\quad d(x,y)<r\right\},\quad\mathcal{S}_r(x)\doteq\left\{y\in M\;\vline\quad d(x,y)=r\right\},\quad\forall x\in M.
$$
The unit tangent sphere in $\mathrm{T}_xM$ will be identified with the unit sphere in $\mathbb{R}^n$ and denoted by $\mathbb{S}$.

For every $x\in M$, the function $d(x,\cdot)$ is a Lipschitz, piecewise smooth function, and for almost every $r\in\mathbb{R}_+$, the sphere $\mathcal{S}_r(x)$ is a piecewise smooth hypersurface. This follows from the fact that the cut locus $\mathcal{C}(x)$ of a point $x\in M$ is a closed subset of measure zero (see Proposition III.3.1 in \cite{Cha06}). Whenever $\mathcal{S}_r(x)$ is piecewise smooth, its surface measure $A_{r,x}$ is given by the restriction of the $(n-1)$-form $*d\,d(x,\cdot)$, and the surface normal $N_{r,x}=\langle \nabla d(x,\cdot),\cdot\rangle$ is defined almost everywhere on $\mathcal{S}_r(x)$. For later use, we denote
\begin{equation}
\mathcal{E}(x)\doteq\left\{r\in\mathbb{R}_+\;\vline\quad A_{r,x}(\mathcal{S}_r(x)\cap\mathcal{C}(x))=0\right\},\quad\left|\mathcal{E}(x)^\complement\right|=0.\label{EDef}
\end{equation}

Often we will fix a point $y\in M$ and work in geodesic spherical coordinates as in \S III.1 of \cite{Cha06}. The Riemannian exponential map based at $y$ will be denoted by $\exp_y:U_y\to M$, where $U_y\subset\mathbb{R}_+\times\mathbb{S}$ is the domain of injectivity. The volume element has the form
$$
d\mathrm{vol}_g(x)=dA_{r,y}(x)dr,\quad dA_{r,y}(x)=\rho_y(r,\varphi)d\varphi,
$$
where $r=d(y,\exp_y(r,\varphi))$ is the radial coordinate and $d\varphi$ is the standard measure of the Euclidean unit $(n-1)$-sphere $\mathbb{S}$. For every $\varphi\in\mathbb{S}$, denote by $c_y(\varphi)\in(0,+\infty]$ the first cut point in the direction $\varphi$, so that $\rho_y(r,\varphi)>0$ for all $r\in(0,c_y(\varphi))$.

\subsection{The decay of the spherical sums of eigenfunctions}

We continue working in the geodesic spherical coordinates based at an arbitrarily fixed point $y\in M$. If $r_0\in[0,+\infty)$ and $\psi\in C(M\setminus\overline{\mathcal{B}_{r_0}(y)},\mathbb{R}_+)\cap L^1(M\setminus\mathcal{B}_{r_0}(y))$ then denote by $\Psi_y:[r_0,+\infty)\to\mathbb{R}_-$ the function
$$
\Psi_y(r)\doteq-\int\limits_{M\setminus\mathcal{B}_r(y)}\psi(x)d\mathrm{vol}_g(x),\quad\forall r\in[r_0,+\infty).
$$
Thus, $\Psi_y$ is the (negative) extraglobular sum of $\psi$ at $r\in[r_0,+\infty)$ around $y$. Obviously, $\Psi_y$ is a negative, continuous, non-decreasing function with
\begin{equation}
\Psi_y(r_0)=-\|\psi\|_{L^1(M\setminus\mathcal{B}_{r_0}(y))},\quad\lim_{r\to+\infty}\Psi(r)=0.\label{PsiLimits}
\end{equation}
It follows from Proposition III.5.1 in \cite{Cha06} that
$$
\Psi_y(r)=-\int\limits_r^{+\infty}\int\limits_{\mathbb{S}\,\cap\frac1tU_y}\psi(\exp_y(t,\varphi))\rho_y(t,\varphi)d\varphi dt,\quad\forall r\in[r_0,+\infty).
$$
$\Psi_y$ is a locally Lipschitz function, and its almost everywhere derivative is denoted by
\begin{equation}
\bar\psi_y\in L^\infty_\mathrm{loc}((r_0,+\infty))\cap L^1((r_0,+\infty)),\label{barpsiProp}
\end{equation}
$$
\bar\psi_y(r)=\frac{d\Psi_y}{dr}(r),\quad\mbox{a.e.}\;r\in[r_0,+\infty).
$$
It is clear that
$$
\bar\psi_y(r)=\int\limits_{\mathcal{S}_r(y)}\psi(x)dA_{r,y}(x)=\int\limits_{\mathbb{S}\,\cap\frac1rU_y}\psi(\exp_y(r,\varphi))\rho_y(r,\varphi)d\varphi,\quad\forall r\in\mathcal{E}(y),
$$
where $\mathcal{E}(y)$ is as in (\ref{EDef}). Therefore, $\bar\psi_y(r)$ is the spherical sum of $\psi$ at $r\in\mathcal{E}(y)$ around $y$.

Our main instrument is the following fact: under natural assumptions, if $\psi$ is a positive integrable eigenfunction of the Laplace-Beltrami operator outside a metric ball, then its extraglobular sums $\Psi_y$ weakly satisfy a second order differential inequality.

\begin{theorem}\label{DiffIneqTheorem} Let $y\in M$, $r_0\in[0,+\infty)$ and
$$
\psi\in C^\infty(M\setminus\overline{\mathcal{B}_{r_0}(y)}\,,\mathbb{R}_+)\cap L^1(M\setminus\mathcal{B}_{r_0}(y))
$$
be a positive integrable eigenfunction of the Laplace-Beltrami operator $\Delta$ for an eigenvalue $\lambda\in\mathbb{R}$,
$$
\Delta\psi(x)=\lambda\psi(x),\quad\forall x\in M\setminus\overline{\mathcal{B}_{r_0}(y)}.
$$
Suppose that the flow of $\nabla\psi$ out of spheres $\mathcal{S}_r(y)$ is asymptotically not strictly positive as $r\to+\infty$,
\begin{equation}
\liminf_{\overset{\scriptstyle r\to+\infty}{r\in\mathcal{E}(y)}}\int\limits_{\mathcal{S}_r(y)}\langle\nabla\psi(x),N_{r,y}(x)\rangle dA_{r,y}(x)\le0.\label{liminfCond}
\end{equation}
If $\mu:(r_0,+\infty)\to\mathbb{R}$ is a measurable function that dominates the spherical mean curvature,
$$
\frac{\partial\ln\rho_y(r,\varphi)}{\partial r}\le\mu(r),\quad\forall r\in(r_0,c_y(\varphi)),\quad\forall\varphi\in\mathbb{S},
$$
then we have
$$
-\Psi_y''+\mu\Psi_y'+\lambda\Psi_y\ge0
$$
in the sense of distributions $C^\infty_c((r_0,+\infty))'$.
\end{theorem}
\begin{proof} Fix $y\in M$ and $r_0\in[0,+\infty)$. By definition, for every $r\in\mathcal{E}(y)$, the sphere $\mathcal{S}_r(y)=\partial\mathcal{B}_r(y)$ is a Lipschitz, almost everywhere smooth surface, which allows us to apply the divergence theorem. Namely, for $\forall s,r\in\mathcal{E}(y)$ with $s<r$,
$$
\lambda(\Psi_y(r)-\Psi_y(s))=\int\limits_{\mathcal{B}_r(y)\setminus\mathcal{B}_s(y)}\lambda\psi(x)d\mathrm{vol}_g(x)=\int\limits_{\mathcal{B}_r(y)\setminus\mathcal{B}_s(y)}\Delta\psi(x)d\mathrm{vol}_g(x)
$$
$$
=\int\limits_{\mathcal{S}_r(y)}\langle\nabla\psi(x),N_{r,y}(x)\rangle dA_{r,y}(x)-\int\limits_{\mathcal{S}_s(y)}\langle\nabla\psi(x),N_{s,y}(x)\rangle dA_{s,y}(x).
$$
This implies that
\begin{equation}
c\doteq\int\limits_{\mathcal{S}_r(y)}\langle\nabla\psi(x),N_{r,y}(x)\rangle dA_{r,y}(x)-\lambda\Psi_y(r)=\mathrm{const},\quad\forall r\in\mathcal{E}(y).\label{PsiGradConst}
\end{equation}
By condition (\ref{liminfCond}), we find a sequence $\{r_k\}_{k=1}^\infty\subset\mathcal{E}(y)$ such that $r_k\xrightarrow[k\to\infty]{}+\infty$ and
$$
c=\lim_{k\to\infty}\left[\,\int\limits_{\mathcal{S}_{r_k}(y)}\langle\nabla\psi(x),N_{r_k,y}(x)\rangle dA_{r_k,y}(x)-\lambda\Psi_y(r_k)\right]\le0,
$$
where we used (\ref{PsiLimits}).

Let us now compute the weak derivative $\Psi_y''=\bar\psi_y'$ in the sense of distributions $C^\infty_c((r_0,+\infty))'$. Take $f\in C^\infty_c((r_0,+\infty))\hookrightarrow C_0([0,+\infty])$, $f\ge0$, and write
$$
\bar\psi_y'(f)=-\int\limits_0^{+\infty}\bar\psi_y(r)f'(r)dr=-\int\limits_0^{+\infty}\int\limits_{\mathbb{S}\,\cap\frac1rU_y}\psi(\exp_y(r,\varphi))\rho_y(r,\varphi)f'(r)d\varphi dr
$$
$$
=-\int\limits_{\mathbb{S}}\int\limits_0^{c_y(\varphi)}\psi(\exp_y(r,\varphi))\rho_y(r,\varphi)f'(r)drd\varphi
$$
$$
=-\int\limits_{\mathbb{S}}\biggl[\psi(\exp_y(c_y(\varphi),\varphi))\rho_y(c_y(\varphi),\varphi)f(c_y(\varphi))-\int\limits_0^{c_y(\varphi)}\frac{\partial\psi}{\partial r}(\exp_y(r,\varphi))\rho_y(r,\varphi)f(r)dr
$$
$$
-\int\limits_0^{c_y(\varphi)}\psi(\exp_y(r,\varphi))\frac{\partial\rho_y}{\partial r}(r,\varphi)f(r)dr\biggr]d\varphi
$$
\begin{equation}
\le\int\limits_{\mathbb{S}}\int\limits_0^{c_y(\varphi)}\frac{\partial\psi}{\partial r}(\exp_y(r,\varphi))\rho_y(r,\varphi)f(r)drd\varphi+\int\limits_{\mathbb{S}}\int\limits_0^{c_y(\varphi)}\psi(\exp_y(r,\varphi))\frac{\partial\rho_y}{\partial r}(r,\varphi)f(r)drd\varphi,\label{Psi''Eq}
\end{equation}
where evaluation at $c_y(\varphi)$ is understood in the sense of the limit $r\to c_y(\varphi)-$. Note that
$$
\int\limits_{\mathbb{S}}\int\limits_0^{c_y(\varphi)}\frac{\partial\psi}{\partial r}(\exp_y(r,\varphi))\rho_y(r,\varphi)f(r)drd\varphi=\int\limits_0^{+\infty}f(r)\int\limits_{\mathbb{S}\,\cap\frac1rU_y}\frac{\partial\psi}{\partial r}(\exp_y(r,\varphi))\rho_y(r,\varphi)d\varphi dr
$$
\begin{equation}
=\int\limits_0^{+\infty}f(r)\int\limits_{\mathcal{S}_r(y)}\langle\nabla\psi(x),N_{r,y}(x)\rangle dA_{r,y}(x)dr\le\lambda\int\limits_0^{+\infty}f(r)\Psi_y(r)dr,\label{Part1Eq}
\end{equation}
where we used (\ref{PsiGradConst}) with $c\le0$. On the other hand,
$$
\int\limits_{\mathbb{S}}\int\limits_0^{c_y(\varphi)}\psi(\exp_y(r,\varphi))\frac{\partial\rho_y}{\partial r}(r,\varphi)f(r)drd\varphi=\int\limits_0^{+\infty}f(r)\int\limits_{\mathbb{S}\,\cap\frac1rU_y}\psi(\exp_y(r,\varphi))\frac{\partial\rho_y}{\partial r}(r,\varphi)d\varphi dr
$$
\begin{equation}
\le\int\limits_0^{+\infty}f(r)\mu(r)\int\limits_{\mathbb{S}\,\cap\frac1rU_y}\psi(\exp_y(r,\varphi))\rho_y(\exp_y(r,\varphi))d\varphi dr=\int\limits_0^{+\infty}f(r)\mu(r)\bar\psi_y(r)dr.\label{Part2Eq}
\end{equation}
Combining (\ref{Psi''Eq}), (\ref{Part1Eq}) and (\ref{Part2Eq}), we arrive at
$$
\Psi_y''(f)\le\lambda\Psi_y(f)+\mu\Psi_y'(f).
$$
Since $f\ge0$ was arbitrary, this implies
$$
-\Psi_y''+\mu\Psi_y'+\lambda\Psi_y\ge0
$$
in the sense of $C^\infty_c((r_0,+\infty))'$, as desired.
\end{proof}

Next we will show that the differential inequality obtained in Theorem \ref{DiffIneqTheorem} implies decay estimates.

\begin{proposition}\label{MainProp} Let $\Psi:[r_0,+\infty)\to\mathbb{R}_-$, $r_0\in\mathbb{R}$, be a locally Lipschitz, non-decreasing function with
$$
\Psi(r)\xrightarrow[r\to+\infty]{}0.
$$
Let further $\mu:[r_0,+\infty)\to\mathbb{R}$ be a non-increasing function.
Suppose that
\begin{equation}
-\Psi''+\mu\Psi'+\lambda\Psi\ge0\label{MainPropMainEq}
\end{equation}
in the sense of distributions $C_c^\infty((r_0,+\infty))'$ for some $\lambda>0$. Then
$$
\Psi'(r)\le\left[\operatorname{ess}\liminf_{t\to r_0+}\Psi'(t)+|\Psi(r_0)|\left(\frac{\mu(r_0)}2+\sqrt{\frac{\mu(r_0)^2}4+\lambda}\right)\right]
$$
$$
\times\; e^{-\int\limits_{r_0}^r\left[-\frac{\mu(s)}2+\sqrt{\frac{\mu(s)^2}4+\lambda}\,\right]ds},\quad\mbox{a.e.}\;r\in[r_0,+\infty).
$$
\end{proposition}
\begin{proof} Consider the ordinary differential equation
\begin{equation}
\beta'=\beta^2+\mu\beta-\lambda=(\beta-\beta_-)(\beta-\beta_+)\label{BetaEq}
\end{equation}
for a locally absolutely continuous function $\beta:[r_0,+\infty)\to\mathbb{R}$. Here $\beta_\pm:[r_0,+\infty)\to\mathbb{R}_\pm$ are monotone non-decreasing functions given by
$$
\beta_\pm(r)\doteq-\frac{\mu(r)}2\pm\sqrt{\frac{\mu(r)^2}4+\lambda},\quad\forall r\in[r_0,+\infty).
$$
Let $\bar\beta_-:[r_0,+\infty)\to\mathbb{R}_-$ be the right-continuous (upper semi-continuous) modification of $\beta_-$,
$$
\bar\beta_-(r)\doteq\limsup_{t\to r}\beta_-(t),\quad\forall r\in[r_0,+\infty).
$$
It is clear that $\bar\beta_-(r)=\beta_-(r)$ for all but countably many $r\in[r_0,+\infty)$. Since the equation (\ref{BetaEq}) was originally to be satisfied almost everywhere, one can replace $\beta_-$ by $\bar\beta_-$ in the equation without changing the solutions.

Choose the initial condition $\beta(r_0)=\bar\beta_-(r_0)$. By Carath\'eodory's theorem (see Theorem 5.1 and Theorem 5.2 in Chapter I of \cite{Hal80}), there exists a locally absolutely continuous solution $\beta$ of the equation (\ref{BetaEq}) with this initial condition on a maximal interval $[r_0,a)$, $a\in(r_0,+\infty]$. Firstly, let us prove that
\begin{equation}
\beta(r)\le\bar\beta_-(r),\quad\forall r\in[r_0,a).\label{beta<beta-}
\end{equation}
Assume towards a contradiction that the set $Z\doteq[\beta-\bar\beta_-]^{-1}(\mathbb{R}_+)$ is not empty. Since $\beta$ is continuous and $-\bar\beta_-$ is lower semi-continuous, we have that $\beta-\bar\beta_-$ is lower semi-continuous, and $Z=[\beta-\bar\beta_-]^{-1}(\mathbb{R}_+)\subset[r_0,a)$ is an open subset. Let $b\doteq\inf Z$, so that $(b,c)\subset Z$. Then $\beta(b)\le\bar\beta_-(b)$, or else we would have $b\in Z$. On the other hand, by continuity, there exists $d\in(b,c)$ such that
$$
\bar\beta_-(r)<\beta(r)<\beta_+(b)\le\beta_+(r),\quad\forall r\in(b,d).
$$
It follows that
$$
\beta'(r)=(\beta(r)-\bar\beta_-(r))(\beta(r)-\beta_+(r))<0,\quad\mbox{a.e.}\;r\in(b,d),
$$
implying that
$$
\beta(d)=\beta(b)+\int\limits_b^d\beta'(r)ds<\beta(b)\le\bar\beta_-(b)\le\bar\beta(d),
$$
which is in contradiction with $d\in Z$. This contradiction proves the inequality (\ref{beta<beta-}). Moreover, it shows also that
$$
\beta'(r)=(\beta(r)-\bar\beta_-(r))(\beta(r)-\beta_+(r))\ge0,\quad\mbox{a.e.}\;r\in[r_0,a),
$$
whence
$$
\beta(r)=\beta_0+\int\limits_{r_0}^r\beta'(s)ds\ge\beta_0,\quad\forall r\in[r_0,a).
$$
If $a<+\infty$ then
$$
\beta_0\le\lim\limits_{r\to a-}\beta(r)\le\lim\limits_{r\to a-}\bar\beta_-(r)\le\bar\beta_-(a),
$$
which contradicts the maximality of the existence interval $[r_0,a)$ (apply Theorem 5.2 in Chapter I of \cite{Hal80} with $D=(r_0,+\infty)\times\mathbb{R}$). Thus, we establish that $a=+\infty$, and the solution $\beta$ is well defined on $[r_0,+\infty)$. We restate for the later use that
\begin{equation}
\beta(r)\le\bar\beta_-(r)=\beta_-(r)=-\frac{\mu(r)}2-\sqrt{\frac{\mu(r)^2}4+\lambda},\quad\mbox{a.e.}\; r\in[r_0,+\infty).\label{BetaIneq}
\end{equation}

Now define the locally Lipschitz function $H:[r_0,+\infty)\to\mathbb{R}_-$ by
$$
H(r)=\Psi(r)\,e^{\;\int\limits_{r_0}^r\beta(s)ds},\quad\forall r\in[r_0,+\infty).
$$
It follows from (\ref{MainPropMainEq}) and (\ref{BetaEq}) that
$$
-H''+(2\beta+\mu)H'=e^{\,\int\limits_{r_0}\beta(s)ds}\left[-\Psi''+\mu\Psi'+\lambda\Psi\right]+\Psi e^{\,\int\limits_{r_0}\beta(s)ds}\left[-\beta'+\beta^2+\mu\beta-\lambda\right]
$$
$$
=e^{\,\int\limits_{r_0}\beta(s)ds}\left[-\Psi''+\mu\Psi'+\lambda\Psi\right],
$$
understood in the space $\mathrm{W}^{-1,\infty}_\mathrm{loc}((r_0,+\infty))$. Since $C^\infty_c((r_0,+\infty))\subset\mathrm{W}^{1,1}_c((r_0,+\infty))$ is dense, the inequality (\ref{MainPropMainEq}) is true also in $\mathrm{W}^{-1,\infty}_\mathrm{loc}((r_0,+\infty))$, whence
$$
-H''+(2\beta+\mu)H'\ge0
$$
in the sense of $\mathrm{W}^{-1,\infty}_\mathrm{loc}((r_0,+\infty))$. Note that
$$
-\left[H'e^{-\int\limits_{r_0}[2\beta(s)+\mu(s)]ds}\right]'=e^{-\int\limits_{r_0}[2\beta(s)+\mu(s)]}\left[-H''+(2\beta+\mu)H'\right]\ge0
$$
in the sense of $\mathrm{W}^{-1,\infty}_\mathrm{loc}((r_0,+\infty))$. Thus,
$$
-\left[H'e^{-\int\limits_{r_0}[2\beta(s)+\mu(s)]ds}\right]'
$$
is a Radon measure (see Theorem 2.1.7 in \cite{HorI}), and hence
$$
H'e^{-\int\limits_{r_0}[2\beta(s)+\mu(s)]ds}
$$
equals almost everywhere a monotone non-increasing function. This implies
$$
H'(r)\le\operatorname{ess}\liminf_{t\to r_0+}H'(t)\cdot e^{\;\int\limits_{r_0}^r\left[2\beta(s)+\mu(s)\right]ds},\quad\mbox{a.e.}\;r\in[r_0,+\infty).
$$
In particular, since we have
$$
H'(r)=\Psi'(r)\,e^{\;\int\limits_{r_0}^r\beta(s)ds}+\beta(r)\Psi(r)\,e^{\;\int\limits_{r_0}^r\beta(s)ds},\quad\mbox{a.e.}\;r\in[r_0,+\infty),
$$
it follows that (recall that $\beta\Psi>0$)
$$
\Psi'(r)\le\left[\operatorname{ess}\liminf_{t\to r_0+}\Psi'(t)+\beta_0\Psi(r_0)\right]\,e^{\;\int\limits_{r_0}^r\left[\beta(s)+\mu(s)\right]ds}
$$
$$
\le\left[\operatorname{ess}\liminf_{t\to r_0+}\Psi'(t)+|\Psi(r_0)|\left(\frac{\mu(r_0)}2+\sqrt{\frac{\mu(r_0)^2}4+\lambda}\right)\right]\,e^{-\int\limits_{r_0}^r\left[-\frac{\mu(s)}2+\sqrt{\frac{\mu(s)^2}4+\lambda}\right]ds},
$$
$$
\mbox{a.e.}\;r\in[r_0,+\infty),
$$
where we used (\ref{BetaIneq}) in the last step. The proof is complete.
\end{proof}

\section{Main results}

In this chapter, we begin deriving our main results. Throughout the chapter, we will maintain the following assumptions about the Riemannian manifold $(M,g)$.

\begin{assumptions}\label{Assumptions} We suppose that the Riemannian manifold $(M,g)$ is such that:
\begin{itemize}

\item[(i)] $(M,g)$ is complete and non-compact;

\item[(ii)] the Ricci curvature is bounded from below,
$$
\mathrm{Ric}\ge-(n-1)\kappa,\quad\kappa\in[0,+\infty).
$$

\end{itemize}
\end{assumptions}

It is worth noting that a manifold of bounded geometry automatically satisfies the Assumptions \ref{Assumptions}:(ii) .

\subsection{An upper bound for the spherical sums of the resolvent kernel}

As the reader may have already expected, the upper bound for the spherical sums of the resolvent kernel will be derived by combining Theorem \ref{DiffIneqTheorem} and Proposition \ref{MainProp}. Note that Theorem \ref{DiffIneqTheorem} requires a function $\mu$ that dominates the spherical mean curvature $\partial_r\ln\rho_y$, whereas Proposition \ref{MainProp} works with a function $\mu$ that is monotone non-increasing. These two requirements converge into the following definition.

Let us define for every $y\in M$ the monotone non-increasing function $\mu_y:\mathbb{R}_+\to\mathbb{R}$ by
\begin{equation}
\mu_y(r)\doteq\sup_{t\in[r,+\infty)}\sup_{\varphi\in\mathbb{S}\,\cap\frac1tU_y}\frac{\partial\ln\rho_y(t,\varphi)}{\partial t},\quad\forall r\in\mathbb{R}_+.\label{VolUpperBound}
\end{equation}
Note that since $M$ is not compact, $\mathbb{S}\cap\frac1tU_y\neq\emptyset$ for all $t\in\mathbb{R}_+$. By Bishop's theorem (Theorem III.4.3 in \cite{Cha06}), we have due to Assumptions \ref{Assumptions}:(ii) that
\begin{equation}
\mu_y(r)\le\frac{(n-1)\sqrt{\kappa}}{\tanh(\sqrt{\kappa}r)},\quad\forall r\in\mathbb{R}_+,\label{muComparison}
\end{equation}
with the understanding that, for $\kappa=0$, the limit is taken as $\kappa\to0+$ to yield $\frac{n-1}r$ on the right-hand side.

We retain all notations established in the previous sections. Let us begin with a small but very useful lemma.

\begin{lemma}\label{SalvationLemma} For every $y\in M$ and $\lambda\in\mathbb{R}_+$, we have
$$
\int\limits_{\mathcal{S}_r(y)}\langle\nabla_xK_\lambda(x,y),N_{r,y}(x)\rangle dA_{r,y}(x)<0,\quad\forall r\in\mathcal{E}(y).
$$
\end{lemma}
\begin{proof} Let $y\in M$ and $\lambda,r_0\in\mathbb{R}_+$ be given. We use the relation (\ref{KpRel}) to express the resolvent kernel in terms of the heat kernel,
$$
\nabla_xK_\lambda(x,y)=\nabla_x\int\limits_0^{+\infty}e^{-\lambda t}p(x,y,t)dt
$$
$$
=\int\limits_0^{+\infty}e^{-\lambda t}\nabla_xp(x,y,t)dt,\quad\forall x\in M\setminus\mathcal{B}_{r_0}(y),
$$
where the second equality is justified by the fact that the integral on the second line converges uniformly for $d(x,y)\ge r_0$, due to Theorem 6 in \cite{Dav89}. Consequently, we have
$$
\int\limits_{\mathcal{S}_r(y)}\langle\nabla_xK_\lambda(x,y),N_{r,y}(x)\rangle dA_{r,y}(x)=\int\limits_{\mathcal{S}_r(y)}\int\limits_0^{+\infty}e^{-\lambda t}\langle\nabla_xp(x,y,t),N_{r,y}(x)\rangle dtdA_{r,y}(x)
$$
\begin{equation}
=\int\limits_0^{+\infty}e^{-\lambda t}\int\limits_{\mathcal{S}_r(y)}\langle\nabla_xp(x,y,t),N_{r,y}(x)\rangle dA_{r,y}(x)dt,\quad\forall r\in\mathcal{E}(y)\cap[r_0,+\infty),\label{gradKtogradpEq}
\end{equation}
where the change of order of integration is again justified by uniform convergence. However, by the divergence theorem, we have
$$
\int\limits_{\mathcal{S}_r(y)}\langle\nabla_xp(x,y,t),N_{r,y}(x)\rangle dA_{r,y}(x)=\int\limits_{\mathcal{B}_r(y)}\Delta_xp(x,y,t)d\mathrm{vol}_g(x)
$$
$$
=\int\limits_{\mathcal{B}_r(y)}\frac{\partial p}{\partial t}(x,y,t)d\mathrm{vol}_g(x)=\frac{\partial}{\partial t}\int\limits_{\mathcal{B}_r(y)}p(x,y,t)d\mathrm{vol}_g(x),\quad\forall t\in\mathbb{R}_+,\quad\forall r\in\mathcal{E}(y).
$$
We will now apply integration by parts to write
$$
\int\limits_0^{+\infty}e^{-\lambda t}\int\limits_{\mathcal{S}_r(y)}\langle\nabla_xp(x,y,t),N_{r,y}(x)\rangle dA_{r,y}(x)dt
$$
$$
=\int\limits_0^{+\infty}e^{-\lambda t}\frac{\partial}{\partial t}\int\limits_{\mathcal{B}_r(y)}p(x,y,t)d\mathrm{vol}_g(x)dt=-\lim_{t\to0+}e^{-\lambda t}\int\limits_{\mathcal{B}_r(y)}p(x,y,t)d\mathrm{vol}_g(x)
$$
\begin{equation}
+\lambda\int\limits_0^{+\infty}e^{-\lambda t}\int\limits_{\mathcal{B}_r(y)}p(x,y,t)d\mathrm{vol}_g(x)dt<0,\quad\forall r\in\mathcal{E}(y),\label{FlowNegativeEq}
\end{equation}
where we made use of the stochastic completeness of $M$ and
$$
\lim_{t\to0+}\int\limits_{\mathcal{B}_r(y)}p(x,y,t)d\mathrm{vol}_g(x)=1.
$$
To be pedantic, the last equality can be seen as follows. Let $f\in C_c(\mathcal{B}_r(y),\mathbb{R}_+)$ be a positive continuous function, supported in $\mathcal{B}_r(y)$, with $f(y)=1$ and $f\le 1$. Then
$$
1\ge\int\limits_{\mathcal{B}_r(y)}p(x,y,t)d\mathrm{vol}_g(x)\ge\int\limits_{\mathcal{B}_r(y)}p(x,y,t)f(x)d\mathrm{vol}_g(x)=e^{t\Delta}f(y)\xrightarrow[t\to0+]{}f(y)=1,
$$
by Theorem 7.16 of \cite{Gri09}. Since $r_0\in\mathbb{R}_+$ was arbitrary, from (\ref{gradKtogradpEq}) and (\ref{FlowNegativeEq}) we arrive at the desired assertion.
\end{proof}

Finally, we are in a position to establish the decay estimates for the spherical sums of the resolvent kernel.

\begin{theorem}\label{MainTheorem} Let $(M,g)$ be a Riemannian manifold satisfying the Assumptions \ref{Assumptions}. For every $\lambda\in\mathbb{R}_+$, let $K_\lambda$ be the resolvent kernel, and for every $y\in M$, let $\mu_y$ be as in (\ref{VolUpperBound}). Then for every $r_0\in\mathbb{R}_+$ and $y\in M$,
$$
\int\limits_{\mathcal{S}_r(y)}K_\lambda(x,y)dA_{r,y}(x)\le C_{r_0,\lambda,\kappa,n}(K_\lambda(\cdot,y))\;e^{-\int\limits_{r_0}^r\left[-\frac{\mu_y(s)}2+\sqrt{\frac{\mu_y(s)^2}4+\lambda}\,\right]ds}
$$
\begin{equation}
\le C_{r_0,\lambda,\kappa,n}(K_\lambda(\cdot,y))\;e^{-\alpha_{r_0,\lambda,\kappa}(r-r_0)},\quad\mbox{a.e.}\;r\in[r_0,+\infty),\label{MainEstimate}
\end{equation}
where
\begin{equation}
C_{r_0,\lambda,\kappa,n}(K_\lambda(\cdot,y))\doteq\operatorname{ess}\liminf_{t\to r_0+}\int\limits_{\mathcal{S}_t(y)}K_\lambda(x,y)dA_{t,y}(x)+\frac1{\alpha_{r_0,\lambda,\kappa}},
\end{equation}
and
\begin{equation}
\alpha_{r_0,\lambda,\kappa}\doteq-\frac{(n-1)\sqrt{\kappa}}{2\tanh(\sqrt{\kappa}r_0)}+\sqrt{\frac{(n-1)^2\kappa}{4\tanh^2(\sqrt{\kappa}r_0)}+\lambda}.
\end{equation}
\end{theorem}
\begin{proof} Take any $y\in M$, set $\psi\doteq K_\lambda(\cdot,y)$ and let $r_0\in\mathbb{R}_+$ be arbitrary. That $\psi\in C^\infty(M\setminus\overline{\mathcal{B}_{r_0}(y)}\,,\mathbb{R}_+)\cap L^1(M\setminus\mathcal{B}_{r_0}(y))$ and $\Delta\psi=\lambda\psi$ in $M\setminus\overline{\mathcal{B}_{r_0}(y)}$ is readily known by (\ref{KlambdaSmooth}), (\ref{IntKEq}) and (\ref{KEigenfunc}). Meanwhile, the condition (\ref{liminfCond}) is satisfied thanks to Lemma \ref{SalvationLemma}. Therefore, Theorem \ref{DiffIneqTheorem} applies, giving
$$
-\Psi_y''+\mu\Psi_y'+\lambda\Psi_y\ge0
$$
in the sense of distributions $C^\infty_c((r_0,+\infty))'$. Applying Proposition \ref{MainProp} to $\Psi_y$ as above and $\mu_y$ as in (\ref{VolUpperBound}) we obtain the upper bound (\ref{MainEstimate}) with a constant
$$
C\doteq\operatorname{ess}\liminf_{t\to r_0+}\Psi_y'(t)+|\Psi_y(r_0)|\left(\frac{\mu_y(r_0)}2+\sqrt{\frac{\mu_y(r_0)^2}4+\lambda}\right).
$$
From (\ref{IntKEq}) and (\ref{muComparison}) we find that
$$
|\Psi_y(r_0)|\left(\frac{\mu_y(r_0)}2+\sqrt{\frac{\mu_y(r_0)^2}4+\lambda}\,\right)\le\frac1\lambda\left[\frac{(n-1)\sqrt{\kappa}}{2\tanh(\sqrt{\kappa}r_0)}+\sqrt{\frac{(n-1)^2\kappa}{4\tanh^2(\sqrt{\kappa}r_0)}+\lambda}\,\right]
$$
$$
=\left[-\frac{(n-1)\sqrt{\kappa}}{2\tanh(\sqrt{\kappa}r_0)}+\sqrt{\frac{(n-1)^2\kappa}{4\tanh^2(\sqrt{\kappa}r_0)}+\lambda}\,\right]^{-1},
$$
which completes the proof.
\end{proof}

Although the coefficient $C_{r_0,\lambda,\kappa,n}$ in Theorem \ref{MainTheorem} has no ambitions of being optimal (our primary goal was the exponential decay rate), for some applications it may nevertheless be useful to have an upper estimate for it.

\begin{proposition}\label{CoeffUBoundProp} Let the Assumptions \ref{Assumptions} be satisfied, together with a uniform volume-non-collapsing condition
$$
\mathrm{vol}_g(\mathcal{B}_b(x))\ge B,\quad\forall x\in M,
$$
for some $B,b\in\mathbb{R}_+$. Then for all $r_0\in\mathbb{R}_+$ and $y\in M$,
$$
\operatorname{ess}\liminf_{r\to r_0+}\int\limits_{\mathcal{S}_r(y)}K_\lambda(x,y)dA_{r,y}(x)
$$
$$
\le\frac{n\,4^ne^{n+n(n-1)\kappa b^2}V_\kappa(b)}{B}\left(\frac{\sinh(\sqrt{\kappa}r_0)}{\sqrt{\kappa}}\right)^{n-1}
$$
$$
\times\left[\frac{c_{\frac{n}2}(6\sqrt{3})^{n-2}}{r_0^{n-2}}\frac{\log^{n=2}\left(e+\frac{108b^2}{r_0^2}\right)}{\left(1+\frac{\sqrt{\lambda+3E}\,r_0}{3\sqrt{3}}\right)^{\frac{3-n}2}}e^{-\frac{\sqrt{\lambda+3E}\,r_0}{3\sqrt{3}}}+\frac{e^{-\lambda\min\left\{\frac{r_0^2}9,b^2\right\}}}{\lambda \min\left\{\frac{r_0^n}{3^n},b^n\right\}}\right],
$$
where
$$
V_\kappa(b)=\frac{2\pi^{\frac{n}2}}{\Gamma(\frac{n}2)}\int\limits_0^b\left(\frac{\sinh(\sqrt{\kappa}s)}{\sqrt{\kappa}}\right)^{n-1}ds
$$
and
$$
\log^{n=2}z=\begin{cases}
\log z & \mbox{if}\quad n=2;\\
1 & \mbox{if}\quad n>2.
\end{cases}
$$
\end{proposition}
\begin{proof} If we denote $V_x(r)\doteq\mathrm{vol}_g(\mathcal{B}_r(x))$, then by comparison theorems (Theorem III.4.5 in \cite{Cha06}) we have
\begin{equation}
V_x(r)\ge\frac{V_x(b)}{V_\kappa(b)}V_\kappa(r)\ge\frac{\omega_nB}{V_\kappa(b)}r^n,\quad\forall r\in(0,b],\quad\forall x\in M,\quad\omega_n=\frac{c_{n-1}}n,\label{VxrLowerBound}
\end{equation}
where $V_\kappa(r)$ is the equivalent of $V_x(r)$ in the space of constant curvature $\kappa$,
$$
V_\kappa(r)=c_{n-1}\int\limits_0^r\left(\frac{\sinh(\sqrt{\kappa}s)}{\sqrt{\kappa}}\right)^{n-1}ds,\quad\forall r\in\mathbb{R}_+,\quad c_{n-1}=\frac{2\pi^{\frac{n}2}}{\Gamma(\frac{n}2)},
$$
(see formula (III.4.1) in \cite{Cha06}).

Now we will follow the discussion of heat kernel bounds in \cite{Dav93}. Carefully following the constants in Theorem 5.3.5 of \cite{Dav89a}, and choosing $\alpha=2$ and $s=t$ for convenience, Theorem 2 of \cite{Dav93} gives
\begin{equation}
p(x,y,t)\le\frac{4^ne^n}{\sqrt{V_x(\sqrt{t})V_y(\sqrt{t})}}e^{n(n-1)\kappa t-3Et-\frac{(d(x,y)-2\sqrt{t})_+^2}{12t}},\quad\forall x,y\in M,\quad\forall t\in\mathbb{R}_+.\label{HeatKernelEst}
\end{equation}
For $x=y$ and $t=T\in\mathbb{R}_+$ we find
$$
p(x,x,T)\le\frac{4^ne^n}{V_x(\sqrt{T})}e^{n(n-1)\kappa T-3ET},\quad\forall x\in M,
$$
so that (see Exercises 7.21 and 7.22 in \cite{Gri09})
$$
p(x,y,t)\le\sqrt{p(x,x,t)p(y,y,t)}\le\sqrt{p(x,x,T)p(y,y,T)}
$$
$$
\le\frac{4^ne^n}{\sqrt{V_x(\sqrt{T})V_y(\sqrt{T})}}e^{n(n-1)\kappa T-3ET},\quad\forall x,y\in M,\quad\forall t\in[T,+\infty),\quad\forall T\in\mathbb{R}_+.
$$
If we denote
$$
T_{x,y,b}\doteq\min\left\{\frac{d(x,y)^2}9,b^2\right\},\quad\forall x,y\in M,
$$
then, taking $T=T_{x,y,b}$, by (\ref{VxrLowerBound}) we have
$$
p(x,y,t)\le\frac{4^ne^{n+n(n-1)\kappa b^2}V_\kappa(b)}{\omega_nB}\,\cdot\,\frac1{T_{x,y,b}^{\frac{n}2}},\quad\forall x,y\in M,\quad\forall t\in[T_{x,y,b},+\infty),
$$
with the acknowledgement that the statement is void for $x=y$. On the other hand, from (\ref{HeatKernelEst}) we can obtain the short-time estimate
$$
p(x,y,t)\le\frac{4^ne^{n+n(n-1)\kappa b^2}V_\kappa(b)}{\omega_nB}\,\cdot\,\frac{e^{-3Et-\frac{d(x,y)^2}{108t}}}{t^{\frac{n}2}},\quad\forall t\in\left(0,T_{x,y,b}\right],\quad\forall x,y\in M,
$$
with the understanding that the statement is a tautology for $x=y$. It follows that
$$
K_\lambda(x,y)=\int\limits_0^{+\infty}e^{-\lambda t}p(x,y,t)dt=\int\limits_0^{T_{x,y,b}}e^{-\lambda t}p(x,y,t)dt+\int\limits_{T_{x,y,b}}^{+\infty}e^{-\lambda t}p(x,y,t)dt\le
$$
$$
\frac{4^ne^{n+n(n-1)\kappa b^2}V_\kappa(b)}{\omega_nB}\left[\int\limits_0^{T_{x,y,b}}\frac{e^{-(\lambda+3E)t-\frac{d(x,y)^2}{108t}}}{t^{\frac{n}2}}dt+\frac1{T_{x,y,b}^{\frac{n}2}}\,\int\limits_{T_{x,y,b}}^{+\infty}e^{-\lambda t}dt\right]
$$
$$
\le\frac{4^ne^{n+n(n-1)\kappa b^2}V_\kappa(b)}{\omega_nB}\left[\frac1{b^{n-2}}\int\limits_0^1\frac{e^{-(\lambda+3E)b^2t-\frac{d(x,y)^2}{108b^2t}}}{t^{\frac{n}2}}dt+\frac{e^{-\lambda T_{x,y,b}}}{\lambda T_{x,y,b}^{\frac{n}2}}\right],
$$
$$
\le\frac{4^ne^{n+n(n-1)\kappa b^2}V_\kappa(b)}{\omega_nB}\left[\frac{c_{\frac{n}2}(6\sqrt{3})^{n-2}}{d(x,y)^{n-2}}\frac{\log^{n=2}\left(e+\frac{108b^2}{d(x,y)^2}\right)}{\left(1+\frac{\sqrt{\lambda+3E}\,d(x,y)}{3\sqrt{3}}\right)^{\frac{3-n}2}}e^{-\frac{\sqrt{\lambda+3E}\,d(x,y)}{3\sqrt{3}}}+\frac{e^{-\lambda T_{x,y,b}}}{\lambda T_{x,y,b}^{\frac{n}2}}\right],
$$
for all $x,y\in M$, were we used the estimate (\ref{IncompleteIntabcEst}) from the Appendix. Thus,
$$
\sup_{x\in\mathcal{S}_r(y)}K_\lambda(x,y)\le\frac{4^ne^{n+n(n-1)\kappa b^2}V_\kappa(b)}{\omega_nB}
$$
$$
\times\left[\frac{c_{\frac{n}2}(6\sqrt{3})^{n-2}}{r^{n-2}}\frac{\log^{n=2}\left(e+\frac{108b^2}{r^2}\right)}{\left(1+\frac{\sqrt{\lambda+3E}\,r}{3\sqrt{3}}\right)^{\frac{3-n}2}}e^{-\frac{\sqrt{\lambda+3E}\,r}{3\sqrt{3}}}+\frac{e^{-\lambda\min\{\frac{r^2}9,b^2\}}}{\lambda \min\{\frac{r^n}{3^n},b^n\}}\right],\quad\forall y\in M.
$$

Finally, from comparison theorems (Theorem III.4.3 in \cite{Cha06}) we know that
$$
A_{r,y}(\mathcal{S}_r(y))\le c_{n-1}\left(\frac{\sinh(\sqrt{\kappa}r)}{\sqrt{\kappa}}\right)^{n-1},\quad\forall r\in\mathcal{E}(y),\quad\forall y\in M.
$$
All together, we find that, for all $r_0\in\mathbb{R}_+$ and $y\in M$,
$$
\operatorname{ess}\liminf_{r\to r_0+}\int\limits_{\mathcal{S}_r(y)}K_\lambda(x,y)dA_{r,y}(x)=\lim_{\overset{\scriptstyle r\to r_0+}{r\,\in\,\mathcal{E}(y)}}\;\int\limits_{\mathcal{S}_r(y)}K_\lambda(x,y)dA_{r,y}(x)
$$
$$
\le\lim_{\overset{\scriptstyle r\to r_0+}{r\,\in\,\mathcal{E}(y)}}\left[A_{r,y}(\mathcal{S}_r(y))\cdot\sup_{x\in\mathcal{S}_r(y)}K_\lambda(x,y)\right]
$$
$$
\le\frac{n\,4^ne^{n+n(n-1)\kappa b^2}V_\kappa(b)}{B}\left(\frac{\sinh(\sqrt{\kappa}r_0)}{\sqrt{\kappa}}\right)^{n-1}
$$
$$
\times\left[\frac{c_{\frac{n}2}(6\sqrt{3})^{n-2}}{r_0^{n-2}}\frac{\log^{n=2}\left(e+\frac{108b^2}{r_0^2}\right)}{\left(1+\frac{\sqrt{\lambda+3E}\,r_0}{3\sqrt{3}}\right)^{\frac{3-n}2}}e^{-\frac{\sqrt{\lambda+3E}\,r_0}{3\sqrt{3}}}+\frac{e^{-\lambda\min\left\{\frac{r_0^2}9,b^2\right\}}}{\lambda \min\left\{\frac{r_0^n}{3^n},b^n\right\}}\right],
$$
as desired.
\end{proof}

\begin{corollary} Under the assumptions of Proposition \ref{CoeffUBoundProp}, for every $r_0\in\mathbb{R}_+$ there exists a coefficient $C_{r_0,B,b,\kappa,n}\in\mathbb{R}_+$ such that
$$
\operatorname{ess}\liminf_{r\to r_0+}\int\limits_{\mathcal{S}_r(y)}K_\lambda(x,y)dA_{r,y}(x)\le C_{r_0,B,b,\kappa,n}\frac{e^{-\frac{\sqrt{\lambda+3E}\,r_0}6}}\lambda,\quad\forall y\in M.
$$
\end{corollary}
\begin{proof} The proof is straightforward.
\end{proof}

\subsection{A pointwise lower bound on the resolvent kernel}

In this section we will obtain a simple lower bound on the resolvent kernel, which will contribute to our understanding of how optimal our upper bound is.

\begin{proposition}\label{LowerBoundProp} Let $(M,g)$ be a Riemannian manifold satisfying the Assumptions \ref{Assumptions}. For every $\lambda\in\mathbb{R}_+$, let $K_\lambda$ be the resolvent kernel. There exists a positive rational function $P_{\kappa,n,\lambda}$ on $\mathbb{R}_+$ such that
$$
K_\lambda(x,y)\ge P_{\kappa,n,\lambda}(d(x,y))e^{-\left[\frac{(n-1)\sqrt{\kappa}}2+\sqrt{\frac{(n-1)^2\kappa}4+\lambda}\,\right]d(x,y)},\quad\forall x,y\in M.
$$
\end{proposition}
\begin{proof} We will draw on the comparison with the standard space of constant curvature $-(n-1)\kappa$. If $E_\kappa:[0,+\infty)\times\mathbb{R}_+\to\mathbb{R}_+$ is the function such that the heat kernel $p_\kappa$ of the standard space is given by
$$
p_\kappa(x,y,t)=E_\kappa(d_\kappa(x,y),t),
$$
with $d_\kappa$ being the distance function of the standard space, then by Theorem 7 in VIII.3 of \cite{Cha84} we have
$$
p(x,y,t)\ge E_\kappa(d(x,y),t),\quad\forall x,y\in M,\quad\forall t\in\mathbb{R}_+.
$$
If we denote
$$
F_{\kappa,\lambda}(r)\doteq\int\limits_0^{+\infty}e^{-\lambda t}E_\kappa(r,t)dt,
$$
then the resolvent kernel $K_{\kappa,\lambda}$ of the standard space is
$$
K_{\kappa,\lambda}(x,y)=\int\limits_0^{+\infty}e^{-\lambda t}p_\kappa(x,y,t)dt=F_{\kappa,\lambda}(d_k(x,y)),
$$
while our resolvent kernel satisfies
$$
K_\lambda(x,y)=\int\limits_0^{+\infty}e^{-\lambda t}p(x,y,t)\ge F_{\kappa,\lambda}(d(x,y)),\quad\forall x,y\in M.
$$
But it is known (see formulae (\ref{RnK}) and (\ref{HnK}) in the Introduction) that, through the asymptotics of Bessel functions, the following estimate holds for $\forall x,y\in M$,
$$
K_{\kappa,\lambda}(x,y)=F_{\kappa,\lambda}(d_k(x,y))\ge P_{\kappa,n,\lambda}(d_k(x,y))e^{-\left[\frac{(n-1)\sqrt{\kappa}}2+\sqrt{\frac{(n-1)^2\kappa}4+\lambda}\right]d_k(x,y)},
$$
where $P_{\kappa,n,\lambda}$ is a positive rational function on $\mathbb{R}_+$. The assertion is proven.
\end{proof}

\section{An interpretation of results}\label{DiscussionSection}

\paragraph{Point values versus densities:}
As mentioned in the Introduction, pointwise upper bounds for the heat or resolvent kernel on general manifolds available in the literature are not very explicit in terms of the exponential decay rates they provide. In particular, they fail to guarantee the integrability of the resolvent kernel - a fact that is known a priori. On the contrary, for the spherical sums we find relatively explicit exponential decay estimates. In fact, we what find is much stronger: the resolvent kernel is not only integrable on its own, but even against an exponentially growing weight. And there may be good reasons for this effectiveness of our approach; the more geometrically versatile object is the density of the integral kernel, while the kernel itself may ``exchange'' growth/decay rates with the Riemannian density from point to point. The spherical sums, which we derived our estimates for, can be seen as the ``measure-neutral'' integrals of the density over the unit sphere. In theory, if one could show that the density of the resolvent kernel satisfies a ``measure-neutral'' parabolic PDE, for instance, and derive suitable parabolic Harnack inequalities, by the very ``measure-neutral'' nature of the equation, these inequalities might avoid the exponential factors present in usual, scalar parabolic Harnack inequalities on manifolds with exponential volume growth. In that case, one could produce pointwise upper bounds for the resolvent kernel density from our estimates of the spherical sums. But this is, of course, a mere speculation at the moment.

\paragraph{Comparison of exponential rates:}
Let us examine closely the relations between the rate of geometric growth and the exponential decay rates guaranteed by Theorem \ref{MainTheorem}. Fix a point $y\in M$ and denote
$$
\bar\mu_y\doteq\inf_{r\in\mathbb{R}_+}\mu_y(r)=\lim_{r\to+\infty}\mu_y(r)\in[-\infty,(n-1)\kappa],
$$
where we used the comparison (\ref{muComparison}). By L'H\^ospilat's rule,
\begin{equation}
\lim_{r\to+\infty}\frac{\int\limits_{r_0}^r\left[-\frac{\mu_y(s)}2+\sqrt{\frac{\mu_y(s)^2}4+\lambda}\,\right]ds}r=-\frac{\bar\mu_y}2+\sqrt{\frac{\bar\mu_y^2}4+\lambda},\label{ExpDecRate}
\end{equation}
which is the exact (asymptotic) rate of exponential decay that Theorem \ref{MainTheorem} provides. Theorem \ref{MainTheorem} also provides a uniform ($y$-independent) proxy $\alpha_{r_0,\lambda,\kappa}$ for that exponential decay rate, which is an increasing function of $r_0$, with a limit value
$$
\alpha_{\infty,\lambda,\kappa}\doteq\lim_{r_0\to+\infty}\alpha_{r_0,\lambda,\kappa}=\sup_{r_0\in\mathbb{R}_+}\alpha_{r_0,\lambda,\kappa}=-\frac{(n-1)\sqrt{\kappa}}2+\sqrt{\frac{(n-1)^2\kappa}4+\lambda}
$$
given solely in terms of the spectral parameter $\lambda$ and the number $\kappa$ from the lower bound of the Ricci curvature. It is clear that
\begin{equation}
0<\alpha_{\infty,\lambda,\kappa}\le-\frac{\bar\mu_y}2+\sqrt{\frac{\bar\mu_y^2}4+\lambda}\le+\infty,\label{alphaComparison}
\end{equation}
and equality with $+\infty$ on the rightmost side means super-exponential decay. A comparison with the lower bound in Proposition \ref{LowerBoundProp} reveals that any discrepancy between $\alpha_{\infty,\lambda,\kappa}$ and the actual decay rate of spherical sums must awe its existence to the failure of the spherical area growth to live up to the expectations according to the Ricci curvature lower bound. Namely, the precise asymptotic decay rate of the spherical sums of the resolvent is $\alpha_{\infty,\lambda,\kappa}$ if and only if
$$
A_{r,y}(\mathcal{S}_r(y))\simeq e^{(n-1)\kappa r}.
$$
If the area growth rate is any less than this, then the spherical means of the resolvent kernel may have a faster exponential decay, which cannot be detected by the uniform lower bound on the Ricci curvature alone.

Obviously, $\mu_y$ does a better job at capturing the exponential growth of the spherical area, and the exponential decay of the spherical sums of the resolvent kernel, than the uniform $(n-1)\kappa$ (which is a virtue of Theorem \ref{MainTheorem}). However, even $\mu_y$ may not represent the true spherical area growth. For the simplicity of the discussion, suppose that $(M,g)$ is exponential. Then
$$
\frac{\partial A_{r,y}(\mathcal{S}_r(y))}{\partial r}=\int\limits_{\mathbb{S}}\frac{\partial\rho_y(r,\varphi)}{\partial r}d\varphi\le\mu_y(r)A_{r,y}(\mathcal{S}_r(y)),\quad\forall r\in\mathbb{R}_+,
$$
and a priori, there is no reason to expect an equality in the above estimate. Therefore, in a highly anisotropic manifold, the spherical area growth may fall behind the maximal rate given by the mean spherical curvature and captured by $\mu_y$. Let us, however, assume that the manifold is sufficiently isotropic (at least around the base point $y$), so that
$$
\lim_{r\to+\infty}\frac{\partial\ln A_{r,y}(\mathcal{S}_r(y))}{\partial r}=\bar\mu_y.
$$
It follows that
$$
\mu\doteq\lim_{r\to+\infty}\frac{\ln\left[1+A_{r,y}(\mathcal{S}_r(y))\right]}r=\max\{\bar\mu_y,0\},
$$
which is the asymptotic exponential growth rate of the spherical area. Moreover, if we denote by $V_y(r)$ the measure of the ball $\mathcal{B}_r(y)$,
$$
V_y(r)\doteq\int\limits_0^rA_{s,y}(\mathcal{S}_s(y))ds,\quad\forall r\in\mathbb{R}_+,
$$
then one can show that
$$
\lim_{r\to+\infty}\frac{\ln\left[1+V_y(r)\right]}r=\max\{\bar\mu_y,0\}=\mu.
$$
Thus, we recognise $\mu$ as the exponential growth rate of the volume, which justifies the absence of the index $y$. If $\bar\mu_y\ge0$, then for the upper bound of the exponential growth rate of the spherical \textbf{means}
$$
\frac1{A_{r,y}(\mathcal{S}_r(y))}\int\limits_{\mathcal{S}_r(y)}K_\lambda(x,y)dA_{r,y}(x)
$$
we find the value
\begin{equation}
\frac{\mu}2+\sqrt{\frac{\mu^2}4+\lambda},\label{SphericalMeanExpRate}
\end{equation}
exactly as expected heuristically in (\ref{WishfulEstimate}).

\paragraph{The spectral gap $E$ or the volume growth $\mu$:}
Looking at the Davies - Mandouvalos asymptotics (\ref{DaviesMandouvalos}) of the heat kernel on the hyperbolic space, one faces an alternative of interpretations. Since $E=\frac{\mu^2}4=\frac{m^2}4$ on $\mathbb{H}^{m+1}$, one can express $m$ in terms of either $E$ or $\mu$, inducing a corresponding interpretation. Farther ahead, the most precise heat kernel upper bounds for general manifolds, such as Grigor'yan's formula (\ref{GrigoryanEst}), explicitly include $E$ and are rather spectral-theoretic in their derivations. However, in their interpretations, authors commonly attribute decay terms of the form $e^{-\alpha t}$ to the spectral gap $E$, and terms of the form $e^{-\beta r}$ to the geometric volume growth. See, for instance, the discussion enclosing Theorem 5.3 in \cite{Gri94}, or Remark 3.5 in \cite{AnJi99}. Thus, though not at all obvious from the formulae, specialists have recognised that decay contributions of the form $e^{-\beta r}$, even if expressed in terms of $E$, are in fact due to the volume growth $\mu$, and $E$ serves as a mere proxy for $\frac{\mu^2}4$.

Here we argue that, as far as resolvent kernel upper estimates are concerned, all contributions in the exponential decay are purely geometric in nature, leaving the spectral gap $E$ with the modest role of an approximation of $\frac{\bar\mu_y^2}4$. More precisely, as we saw before, heat kernel estimates of the form (\ref{GrigoryanEst}) produce resolvent kernel estimates of the form (\ref{PostGrigoryanEst}). Comparing with formula (\ref{ExpDecRate}), we clearly see the pattern: everywhere $\bar\mu_y$ is represented as $2\sqrt{E}$. As long as $\bar\mu_y=2\sqrt{E}$, it remains a matter of debate which interpretation is conceptually more justified. However, next we will demonstrate a manifold, for which $\bar\mu_y>2\sqrt{E}$, and formula (\ref{ExpDecRate}) gives a strictly better decay than what could potentially be obtained in terms of $E$.

Let $\Gamma\subset\mathbb{H}^{m+1}$ be a torsionless Kleinian group with critical exponent $\delta(\Gamma)>\frac{m}2$, see \cite{Sul86} for all details. Then $M=\mathbb{H}^{m+1}/\Gamma$ is a connected, complete, non-compact and smooth Riemannian manifold of constant sectional curvature $-m$. By Bischop's comparison theorem (Theorem III.4.3 in \cite{Cha06}), the volume density $\rho_y$ of $M$ is exactly the same as that of $\mathbb{H}^{m+1}$. In particular, $M$ is harmonic (i.e., $\rho_y$ is perfectly isotropic) and $\mu_y\equiv m$. On the other hand, by the generalised Elstrodt-Patterson theorem (Theorem 1.1 in \cite{Sul86}), $E=\delta(\Gamma)(m-\delta(\Gamma))<\frac{m^2}4$. However, since $\mu_y\equiv m=(n-1)\kappa$, we have an equality in the middle of the formula (\ref{alphaComparison}), giving a resolvent decay stronger than what would be obtained from Davies - Mandouvalos's $E$-based heat kernel estimates for $M=\mathbb{H}^{m+1}/\Gamma$ as in Theorem 5.4 of \cite{DaMa88}.

This shows that the appearance of $E$ in resolvent kernel decay estimates is spurious: the resolvent kernel decay is essentially geometric. This is in stark contrast with the heat kernel situation, where the long-time behaviour of $p(x,y,t)$ is dictated exactly by $E$, see Theorem 10.24 in \cite{Gri09}. Therefore, optimal resolvent kernel estimates \textbf{cannot} be derived from heat kernel estimates.

\paragraph{Harmonic spaces:}
If the mean curvature $\partial_r\ln\rho_y$ of the spheres $\mathcal{S}_r(y)$ depends only on $r$, then we have a harmonic space, or an open model in the terminology of Cheeger and Yau \cite{ChYa81}. In that case, as shown in Proposition 2.2 of the same work, the heat kernel $p(x,y,t)$ depends only on $r=d(x,y)$, and so does the resolvent kernel $K_\lambda(x,y)$ by formula (\ref{KpRel}). It follows that the resolvent kernel equals its spherical mean,
$$
K_\lambda(x,y)=\frac1{A_{r,y}(\mathcal{S}_r(y))}\int\limits_{\mathcal{S}_r(y)}K_\lambda(x,y)dA_{r,y}(x),
$$
and the estimates of the spherical means become pointwise estimates. As already mentioned in the Introduction, pointwise estimates of the resolvent kernel on simply connected non-compact homogeneous harmonic spaces (i.e., flat, symmetric of rank one, or Damek-Ricci) can be found from the corresponding estimates of the heat kernel, which are already known. We do not know if non-compact harmonic spaces are exhausted by homogeneous onces; if that is not the case, then our estimates will cover the remaining ones.

\subsection*{Acknowledgements}

The author is supported by the FWO Senior Research Grant G022821N, and by the Methusalem programme of the Ghent University Special Research Fund (BOF) (Grant number 01M01021).

\section{Appendix}

For the sake of completeness and the reader's convenience, here we bring a few well-known facts from the theory of Bessel functions. For every $\alpha,\beta\in\mathbb{R}_+$ and $\gamma\in[1,+\infty)$,
\begin{equation}
\int\limits_0^{+\infty}\frac{e^{-\alpha t-\frac{\beta}t}}{t^\gamma}dt=2\left(\frac\alpha\beta\right)^{\frac{\gamma-1}2}\mathrm{K}_{\gamma-1}\left(2\sqrt{\alpha\beta}\right),\label{Intabc}
\end{equation}
where $\mathrm{K}_{\gamma-1}$ is the Bessel's function. In particular, we have the following asymptotics:
\begin{equation}
\mathrm{K}_{\gamma-1}(x)=\sqrt{\frac{\pi}{2x}}e^{-x}\left(1+\mathcal{O}(x^{-1})\right)\quad\mbox{as}\quad x\to+\infty,\label{BesselAsympBig}
\end{equation}
and
\begin{equation}
\mathrm{K}_{\gamma-1}(x)\sim\begin{cases}
\ln\frac2{x}-\wp & \mbox{if} \quad \gamma=1\\
\frac{\Gamma(\gamma-1)}2\left(\frac2{x}\right)^{\gamma-1} & \mbox{if} \quad \gamma>1
\end{cases}\quad\mbox{as}\quad x\to0+,
\end{equation}
where $\wp$ is Euler's constant. With this information, it can be checked that
\begin{equation}
\int\limits_0^1\frac{e^{-\alpha t-\frac{\beta}t}}{t^\gamma}dt\le\frac{\wp_\gamma\,\log^{\gamma=1}\left(e+\frac1\beta\right)}{\beta^{\gamma-1}(1+2\sqrt{\alpha\beta})^{\frac32-\gamma}}e^{-2\sqrt{\alpha\beta}},\quad\forall\alpha,\beta\in\mathbb{R}_+,\label{IncompleteIntabcEst}
\end{equation}
for some constant $\wp_\gamma\in\mathbb{R}_+$. The power $\gamma=1$ means that this term appears only for $\gamma=1$.

\end{document}